\documentclass{article}
\usepackage{hyperref}
\usepackage{amssymb,amsmath,amscd,dsfont}
\usepackage{amsthm}

\newcommand{\cH}{\ensuremath{\mathcal{H}}}

\newcommand{\cL}{\ensuremath{\mathcal{L}}}
\newcommand{\cV}{\ensuremath{\mathcal{V}}}

\newcommand{\cK}{\ensuremath{\mathcal{K}}}

\newcommand{\fg}{\ensuremath{\mathfrak{g}}}

\newcommand{\fk}{\ensuremath{\mathfrak{k}}}

\newcommand{\fK}{\ensuremath{\mathfrak{K}}}

\newcommand{\gau}{\ensuremath{\mathfrak{gau}}}

\newcommand{\Ad}{\ensuremath{\operatorname{Ad}}}
\newcommand{\U}{\ensuremath{\operatorname{U}}}
\newcommand{\ad}{\ensuremath{\operatorname{ad}}}
\newcommand{\Gau}{\ensuremath{\operatorname{Gau}}}
\newcommand{\Aut}{\ensuremath{\operatorname{Aut}}}

\newcommand{\R}{\ensuremath{\mathbb{R}}}
\newcommand{\C}{\ensuremath{\mathbb{C}}}
\newcommand{\K}{\ensuremath{\mathbb{K}}}
\newcommand{\Z}{\ensuremath{\mathbb{Z}}}

\newcommand{\T}{\ensuremath{\mathbb{T}}}
\newcommand{\bv}{{\bf{v}}}

\newcommand{\bV}{\ensuremath{\mathbb{V}}}

\newcommand{\pr}{\ensuremath{\operatorname{pr}}}

\newcommand{\Diff}{\mathop{{\rm Diff}}\nolimits}

\newcommand{\g}{{\mathfrak g}}
\newcommand{\PU}{\mathop{\rm PU{}}\nolimits}
\newcommand{\OO}{\mathop{\rm O{}}\nolimits}
\newcommand{\Conj}{\mathop{\rm Conj}\nolimits}

\newcommand{\der}{\mathop{\rm der}\nolimits}
\newcommand{\Spec}{\mathop{\rm Spec}\nolimits}

\theoremstyle{plain}
\newtheorem{Theorem}{Theorem}[section]

\newtheorem{Proposition}[Theorem]{Proposition}

\theoremstyle{definition}
\newtheorem{Definition}[Theorem]{Definition}

\newtheorem{Remark}[Theorem]{Remark}
\newtheorem{Example}[Theorem]{Example}

\renewcommand{\:}{\colon}
\newcommand{\1}{\mathbf{1}}
\newcommand{\res}{\vert}

\renewcommand{\hat}{\widehat} 
\renewcommand{\phi}{\varphi}
\newcommand{\dd}{{\tt d}}

\newcommand\oline{\overline}

\begin{document}

\title{Covariant central extensions of gauge Lie algebras}
\author{Bas Janssens\footnote{
B.J.\ acknowledges support from the NWO grant 
613.001.214 ``Generalised Lie algebra sheaves".} {} and  
Karl-Hermann Neeb\footnote{K.-H.~Neeb acknowledges support from 
the Centre Interfacultaire Bernoulli (CIB) 
and the  NSF (National Science Foundation) 
for a research visit at the EPFL.} 
}

\maketitle

\begin{abstract}
Motivated by positive energy representations,
we classify those continuous central extensions of the 
compactly supported
gauge Lie algebra that are covariant under a 1-parameter group of transformations
of the base manifold. 
\end{abstract}


\section{Introduction} 
\label{sec:1}

Let $\pi \colon \cK \rightarrow M$ be a locally trivial bundle of 
finite dimensional Lie groups, with 
corresponding Lie algebra bundle $\fK \rightarrow M$.
We assume that the fibres $\fK_{x}$ are semisimple.
The group $G = \Gamma_{c}(\cK)$ of compactly supported sections, called the 
(compactly supported) \emph{gauge group}, is a
locally convex Lie group with Lie algebra $\fg = \Gamma_{c}(\fK)$, the 
(compactly supported) \emph{gauge Lie algebra}.

In representation theory, one often wishes to impose \emph{positive energy conditions}
derived from a distinguished 1-parameter group 
$\gamma_{M} \colon \R \rightarrow \mathrm{Diff}(M)$ of transformations 
of the base. 
A lift $\gamma \colon \R \rightarrow \mathrm{Aut}(\cK)$ of $\gamma_{M}$
induces a 1-parameter family $\alpha \colon \R \rightarrow \Aut(G)$
of automorphisms of the gauge group.
If $D \in \mathrm{der}(\fg)$ is the derivation 
$D(\xi) := \frac{d}{dt}\big|_{t=0} \alpha_{t*}(\xi)$ induced by $\alpha$, then
the semidirect product 
\[
G \rtimes_{\alpha} \R 
\]
is a locally convex Lie group with Lie algebra
\[
\fg \rtimes_{D} \R.
\]
Since $[0 \oplus 1, \xi \oplus 0] = D(\xi)$, we will identify $0\oplus 1$ with 
$D$ and write \break  ${\fg \rtimes_{D} \R} = \fg \rtimes \R D$ accordingly.
In this note, we give a complete classification of the continuous
1-dimensional central extensions 
$\widehat{\fg}$ of $\g \rtimes_D\R$, in other words, 
we determine the continuous second Lie algebra cohomology 
$H^2(\fg \rtimes_D \R, \R)$.

In order to describe the answer, 
write $\bv \in \cV(\cK)$ for the vector field on 
$\cK$ that generates the flow of $\gamma$, and write 
$\pi_*\bv \in \cV(M)$ for its projection to $M$, which generates the flow of 
$\gamma_{M}$. 
Identifying $\xi \in \Gamma_c(\fK)$ with the corresponding 
vertical left invariant vector field $\Xi_\xi$ on $\cK$, 
the action of the derivation $D$ on $\g = \Gamma_{c}(\fK)$ is 
described by $D\xi = L_\bv \xi$. 
For each fibre $\fK_{x}$, 
the universal invariant bilinear form $\kappa$ takes 
values in the $K$-representation $V(\fK_{x})$, 
and $\bV := V(\fK)$ is a flat bundle over $M$. 
In the (important!) special case that $\fK_{x}$ is a compact simple Lie algebra,
$\kappa$ is simply the Killing form with values in $V(\fK_{x}) = \R$,
and $\bV$ is the trivial real line bundle over $M$.
Given a Lie connection $\nabla$ on $\fK$ and a 
closed $\pi_*\bv$-invariant current $\lambda \in \Omega^1_{c}(M,\bV)'$, 
there is a unique $2$-cocycle $\omega_{\lambda,\nabla}$ on $\g \rtimes \R D$ 
with 
\[ \omega_{\lambda,\nabla}(\xi,\eta) = \lambda(\kappa(\xi,\nabla\eta)), \qquad 
\omega_{\lambda,\nabla}(D,\xi) = \lambda(\kappa(L_{\bv}\nabla,\xi)) 
\quad \mbox{ for } \quad \xi, \eta \in \g.\] 
The class $[\omega_{\lambda,\nabla}] \in H^2(\g \rtimes_D \R,\R)$ 
is independent of the choice of $\nabla$.
One of our main results (Theorem~\ref{EqIjkcykel}) 
asserts that the map $\lambda \mapsto [\omega_{\lambda,\nabla}]$
is a linear isomorphism 
from the space of closed, $\pi_*\bv$-invariant, $\bV$-valued currents on $M$ to the 
continuous Lie algebra cohomology 
$H^2(\g \rtimes_D \R,\R)$.

Our motivation for classifying these central extensions
comes from the theory of \emph{projective 
positive energy representations}.
If $G$ is a Lie group with locally convex Lie algebra $\fg$,
and $\alpha \: \R \to \Aut(G)$ is a homomorphism 
defining a smooth $\R$-action on $G$, then the 
semidirect product $G \rtimes_{\alpha}\R$ is again a Lie group,
with Lie algebra $\g \rtimes \R D$.
For every smooth 
projective unitary representation \break $\oline\rho \: G \rtimes_{\alpha}\R \to \PU(\cH)$ 
of $G \rtimes_{\alpha}\R$, there exists a
central Lie group extension $\widehat{G}$ of $G \rtimes_{\alpha}\R$ by 
the circle group $\T$ 
for which $\overline{\rho}$ lifts to a  smooth \emph{linear} unitary 
representation $\rho \colon \widehat{G} \rightarrow \U(\cH)$
 (see \cite{JN15} for details). 
The Lie algebra $\hat\g$ can then 
 be written as 
\begin{equation}
  \label{eq:d-elt}
\hat\g =  \R C \oplus_\omega (\g \rtimes  \R D), 
\end{equation}
where $\omega$ is a Lie algebra 2-cocycle of $\g \rtimes  \R D$.
The Lie bracket is 
\[ [z C + x + tD, z'C + x' + t' D] 
= \omega(x + tD,x'+t'D)C + [x,x'] + tD(x')- t'D(x)\,,\]  
and $\dd\rho(C) = i \1$ by construction. 
We say that $\oline\rho$ is a {\it positive energy representation} if 
the selfadjoint operator $H := i\dd\rho(D)$ has a spectrum which is bounded below. 

In \cite{JN16} we address the problem of classifying the projective positive 
energy representations of the gauge group $G = \Gamma_c(\cK)$,
for the smooth action \break  
$\alpha \: \R \to \Aut(G)$ induced by a smooth 1-parameter group 
$\gamma \: \R \to \Aut(\cK)$ of bundle automorphisms.
We break this problem into the following steps: 

\begin{itemize}
\item[\rm(PE1)] Classify the $1$-dimensional central Lie algebra extensions 
$\widehat{\fg}$ of $\fg \rtimes_{D} \R$.
\item[\rm(PE2)] Determine which central extensions $\widehat{\fg}$ 
fulfill natural positivity conditions imposed by so-called 
Cauchy--Schwarz estimates required for cocycles coming from 
positive energy representations (cf.~\cite{JN16}). 
\item[\rm(PE3)] For those $\widehat{\fg}$, classify the positive energy representations that 
integrate to a representation of a connected Lie group
$\widehat{G}_0$ with Lie algebra $\hat\g$. 
\end{itemize} 

In the present note we completely solve (PE1) for semisimple structure algebras $\fK_{x}$,
thus completing the first step in the classification of projective positive energy representations.
%
%
%

To proceed with (PE2), we assume in \cite{JN16} 
that the vector field $\pi_*\bv$ on $M$ has no 
zeros and generates a periodic flow, hence defines an 
action of the circle group~$\T$ on $M$. Under this assumption 
we then show that for every projective positive energy 
representation $\overline{\rho}$ of $\fg \rtimes \R D$, there exists a 
locally finite set $\Lambda \subseteq M/\T$ of orbits
such that the $\fg$-part of $\dd\overline{\rho}$ 
factors through the restriction homomorphism
\begin{equation}
  \label{eq:rest}
\g = \Gamma_c(\fK) \to \Gamma_c(\fK\res_{\Lambda_M}) \cong  
\bigoplus_{\lambda \in \Lambda} \cL_{\psi_\lambda}(\fk),
\end{equation}
where $\Lambda_{M} \subseteq M$ is the union of the orbits in $M$, and
\[ \cL_\psi(\fk) = \{ \xi \in C^\infty(\R,\fk) \: (\forall t \in \R)\ 
\xi(t+1) = \psi^{-1}(\xi(t))\} \] 
is the loop algebra twisted by a finite order automorphism $\psi \in \Aut(\fk)$. 
As the positive energy representations of covariant loop algebras and 
their central extensions, the Kac--Moody algebras (\cite{Ka85}), 
are well understood (\cite{PS86}), this allows us to solve~(PE3). 
This result contributes in particular to ``non-commutative distribution'' 
program whose goal is a classification of the irreducible unitary representations 
of gauge groups (\cite{A-T93}).

The structure of this paper is as follows.  
After introducing gauge groups, their Lie algebras and one-parameter groups 
of automorphism in Section~\ref{SectionGGA}, we describe in 
Section~\ref{vanseminaarsimpel} a procedure that provides a 
reduction from semisimple to simple structure Lie algebras, 
at the expense of replacing $M$ by a finite covering manifold~$\hat M$. 
In Section~\ref{sec:3.3}, we introduce the flat bundle $\bV$, which is 
used in a crucial way in Section~\ref{GySsCoc} 
for the description of the natural $2$-cocycles on the gauge algebra. 
The first step (PE1) is completely settled in 
Section~\ref{GySsCoc}, where Theorem~\ref{EqIjkcykel}
describes all $1$-dimensional central extensions of the gauge algebra. 

\tableofcontents

\section{Gauge groups and gauge algebras }
\label{SectionGGA}

Let $\cK \rightarrow M$ be a smooth bundle of Lie groups, 
and let  
$\fK \rightarrow M$ be the associated Lie algebra bundle
with fibres $\fK_{x} = \mathrm{Lie}(\cK_{x})$.
If $M$ is connected, then the fibres $\cK_{x}$ of $\cK\rightarrow M$
are all isomorphic to a fixed structure group $K$,
and the fibres $\fK_{x}$ of $\fK$ are isomorphic to its Lie algebra $\fk = \mathrm{Lie}(K)$.

\begin{Definition}\label{def:gaugegroup} (Gauge group) 
The \emph{gauge group} is the group $\Gamma(\cK)$ of smooth sections 
of $\cK \rightarrow M$, and
the \emph{compactly supported gauge group} 
is the group $\Gamma_{c}(\cK)$ of smooth compactly supported sections.
\end{Definition}

\begin{Definition}(Gauge algebra)
The \emph{gauge algebra} is the Fr\'echet-Lie algebra $\Gamma(\fK)$
of smooth sections of $\fK \rightarrow M$, equipped with the pointwise Lie bracket.
The \emph{compactly supported gauge algebra} $\Gamma_{c}(\fK)$
is the LF-Lie algebra of smooth compactly supported sections. 
\end{Definition}

The compactly supported gauge group $\Gamma_{c}(\cK)$ is a locally convex Lie group,
whose Lie algebra is the compactly supported gauge algebra $\Gamma_{c}(\fK)$.

\begin{Proposition}
There exists a unique smooth structure on $\Gamma_{c}(\cK)$
which makes it a 
locally exponential
Lie group with Lie algebra $\Gamma_{c}(\fK)$
and exponential map $\exp \colon \Gamma_{c}(\fK) \rightarrow \Gamma_{c}(\cK)$ 
defined by pointwise exponentiation.
\end{Proposition}
\begin{proof}
It suffices to prove this in the case that $M$ is connected.
Let $V_{\fk}, W_{\fk} \subseteq \fk$ be open, symmetric 0-neighbourhoods such that 
the exponential
$\exp \colon \fk \rightarrow K$ restricts to a diffeomorphism of
$W_{\fk}$ onto its image, $V_{\fk}$ is contained in $W_{\fk}$, and  
$\exp(V_{\fk}) \cdot \exp(V_{\fk}) \subseteq \exp(W_{\fk})$.

Choose a locally finite cover $(U_{i})_{i\in I}$ of $M$
by open trivialising neighbourhoods for $\cK \rightarrow M$, which possesses
a refinement $(C_{i})_{i\in I}$ such that $C_{i} \subset U_{i}$ is compact for all $i\in I$.
Fix local trivialisations
$\phi_{i} \colon {K \times U_{i}} \stackrel{\sim}{\rightarrow}\cK|_{U_{i}}$ of $\cK$, which 
gives
rise to local trivialisations 
$\dd\phi_{i} \colon {\fk \times U_{i}} \stackrel{\sim}{\rightarrow}\fK|_{U_{i}}$
for $\fK$. Define 
$W_{i} := \dd\phi_{i}(U_{i} \times W_{K})$, and set 
\[
W_{\Gamma_{c}(\fK)} := \{\xi \in \Gamma_{c}(\fK)\,;\, \xi(C_{i}) \subseteq W_{i}\;\forall\, i \in I\}\,.
\]
Similarly, $V_{\Gamma_{c}(\fK)}$ is defined in terms of preimages over $C_i$ of 
$V_{i} := \dd\phi_{i}(U_{i} \times V_{K})$, and both 
$V_{\Gamma_{c}(\fK)}$ and $W_{\Gamma_{c}(\fK)}$ are open in $\Gamma_{c}(\fK)$.
Since the pointwise exponential $\exp \colon \Gamma_{c}(\fK) \rightarrow \Gamma_{c}(\cK)$
is a bijection of $W_{\Gamma_{c}(\fK)}$ onto its image
$W_{\Gamma_{c}(\cK)} := \exp(W_{\Gamma_{c}(\fK)})$, the latter 
inherits a smooth structure.
The same goes for its subset $V_{\Gamma_{c}(\cK)} := \exp(V_{\Gamma_{c}(\fK)})$.

Inversion $W_{\Gamma_{c}(\cK)} \rightarrow W_{\Gamma_{c}(\cK)}$ and multiplication
$V_{\Gamma_{c}(\cK)} \times V_{\Gamma_{c}(\cK)} \rightarrow W_{\Gamma_{c}(\cK)}$ are smooth,
and for every $\sigma \in \Gamma_{c}(\cK)$, there exists an open 
0-neighbourhood $W_{\sigma} \subseteq W_{\Gamma_{c}(\fK)}$ such that 
$\mathrm{Ad}_{\sigma} \colon W_{\sigma} \rightarrow W_{\Gamma_{c}}(\fK)$ is smooth.
It therefore follows from \cite[p.14]{Ti83} 
(which generalises to locally convex Lie groups, cf.~\cite[Thm.~II.2.1]{Ne06}),
that $\Gamma_{c}(\cK)$ possesses 
a unique Lie group structure such that for some open 0-neighbourhood 
$U_{\Gamma_{c}(\fK)} \subseteq W_{\Gamma_{c}(\fK)}$,
the image $\exp(U_{\Gamma_{c}(\fK)}) \subseteq \Gamma_{c}(\cK)$ is an open neighbourhood of 
the identity.
\end{proof}

\begin{Example}
If $\cK \rightarrow M$ is a trivial bundle, then the gauge group is
$\Gamma(\cK) = C^{\infty}(M,K)$, and the gauge algebra is 
$\Gamma(\fK) = C^{\infty}(M,\fk)$.
Similarly, we have $\Gamma_{c}(\cK) = C^{\infty}_{c}(M,K)$ and
$\Gamma_{c}(\fK) = C^{\infty}_{c}(M,\fk)$ for their compactly supported versions.
One can thus think of gauge groups as `twisted versions' of the group of smooth $K$-valued functions on $M$.
\end{Example}

The motivating example of a gauge group is the group $\Gau(P)$ of vertical automorphisms
of a principal fibre bundle $\pi \colon P \rightarrow M$ with structure group $K$.

\begin{Example}(Gauge groups from principal bundles)
A \emph{vertical automorphism} of a principal fibre bundle $\pi \colon P \rightarrow M$ 
is a $K$-equivariant diffeomorphism \break $\alpha \colon P \rightarrow P$ such that 
$\pi \circ \alpha = \alpha$. 
The group $\Gau(P)$ of vertical automorphisms is called the 
\emph{gauge group} of $P$.
It is isomorphic to the group
\begin{equation}
 C^\infty(P,K)^K 
:= \{ f \in C^\infty(P,K)\,;\, (\forall p \in P, k \in K)\, 
f(pk) = k^{-1}f(p)k\}\,,
\end{equation}
with isomorphism $C^{\infty}(P,K)^{K} \stackrel{\sim}{\rightarrow} \Gau(P)$
given by $f \mapsto \alpha_{f}$ with $\alpha_{f}(p) = p f(p)$.
In order to interpret $\Gau(P)$ as a gauge group in the sense of 
Definition~\ref{def:gaugegroup}, we 
construct the bundle of groups 
$\Conj(P) \rightarrow M$ with typical fibre $K$.
For an element $k \in K$, we write $c_k(g) = kgk^{-1}$ for the induced
inner automorphism of $K$, and also $\Ad_k \in \Aut(\fk)$ for the corresponding 
automorphism of its Lie algebra $\fk$. Define 
the bundle of groups $\Conj(P)\rightarrow M$ by
\[\Conj(P) := {P \times K/\sim}\,,\] 
where $\sim$ is
the relation $(pk,h) \sim (p, c_{k}(h))$ for $p \in P$ and  $k, h \in K$. 
We then have isomorphisms 
\[
\Gau(P) \simeq C^{\infty}(P,K)^{K} \simeq \Gamma(\Conj(P))\,,
\]
where $f \in C^{\infty}(P,K)^{K}$ corresponds to the section $\sigma_{f} \in \Gamma(\Conj(P))$
defined by $\sigma_{f}(\pi(p)) = [p,f(p)]$ for all $p\in P$.
The bundle of Lie algebras associated to $\Conj(P)$ is
the \emph{adjoint bundle} $\mathrm{Ad}(P) \rightarrow M$, 
defined as the quotient 
\[\Ad(P) := {P\times_{\Ad}\fk}\]
of $P \times \fk$ modulo the relation $(pk,X) \sim (p, \mathrm{Ad}_{k}(X))$ for 
$p \in P$, $X \in \fk$ and $k \in K$. 
The \emph{compactly supported gauge group}
$\mathrm{Gau}_{c}(P)\subseteq \Gau(P)$ 
is the group of vertical bundle automorphisms of $P$ 
that are trivial outside the preimage
of some compact set in $M$.
Since it is isomorphic to $\Gamma_{c}(\Conj(P))$, it is a locally convex Lie group 
with Lie algebra $\gau_{c}(P) = \Gamma_{c}(\Ad(P))$.
\end{Example}

\begin{Remark}
Gauge groups arise in field theory, as groups of transformations of the space of 
principal connections on $P$ (the gauge fields). 
If the space-time manifold $M$ is not compact, then one imposes boundary conditions on the 
gauge fields at infinity. Depending on how one does this,
the group $\mathrm{Gau}(P)$ may be too big to preserve the set of admissible gauge fields.
One then expects the group of remaining gauge transformations to at least contain 
$\mathrm{Gau}_{c}(P)$, or perhaps even some larger Lie group of gauge transformations 
specified by a decay condition at infinity 
(cf.\ \cite{Wa10, Go04}). 
\end{Remark}

An \emph{automorphism} of $\pi \colon \cK \rightarrow M$ is a pair $(\gamma, \gamma_{M}) \in \Diff(\cK)\times \Diff(M)$
with $\pi \circ \gamma = \gamma_{M}\circ \pi$, such that for each fibre $\cK_{x}$, the map
$\gamma|_{\cK_{x}} \colon \cK_{x} \rightarrow \cK_{\gamma_{M}(x)}$ is a group homomorphism.
Since $\gamma_{M}$ is determined by $\gamma$, we will omit it from the notation.
We denote the group of automorphisms of $\cK$ by $\Aut(\cK)$.

\begin{Definition}(Geometric $\R$-actions)
In the context of gauge groups, we will be interested in $\R$-actions
$\alpha \colon \R \rightarrow \Aut(\Gamma(\cK))$ that are of \emph{geometric} type.
These are derived from
a 1-parameter group
$\gamma \colon \R \rightarrow \mathrm{Aut}(\cK)$ 
by
\begin{equation}
\alpha_{t}(\sigma) := \gamma_{t} \circ \sigma \circ \gamma_{M, t}^{-1}\,.
\end{equation}
\end{Definition}

\begin{Remark}
If $\cK$ is of the form $\Ad(P)$ for a principal fibre bundle $P\rightarrow M$, 
then a 1-parameter group of automorphisms of $P$ induces a 1-parameter 
group of automorphisms of $\cK$.
If we think of the induced diffeomorphisms $\gamma_{M}(t) \in \mathrm{Diff}(M)$
as time translations, then the automorphisms of $P$ 
encode the time translation behaviour
of the gauge fields.
\end{Remark}

The 1-parameter group $\alpha \colon \R \rightarrow \Aut(\Gamma(\cK))$ of group automorphisms
differentiates to a 1-parameter group $\beta \colon \R \rightarrow \Aut(\Gamma(\fK))$
of Lie algebra automorphisms given by 
\begin{equation}\label{LAaut1par}
\beta_{t}(\xi) = \frac{\partial}{\partial \varepsilon}\Big|_{\varepsilon=0} \gamma_{t}\circ e^{\varepsilon \xi} \circ \gamma_{M,t}^{-1}\,.
\end{equation}
The corresponding derivation $D := \frac{\partial}{\partial t} \big|_{t=0} \beta_{t}$ of $\Gamma(\fK)$ can be described
in terms of the infinitesimal generator $\bv \in \mathfrak{X}(\cK)$ of $\gamma$, given by 
$\bv := \frac{\partial}{\partial t}\big|_{t=0} \gamma_{t}$. 
We identify $\xi \in \Gamma(\fK)$ with the vertical, left invariant vector field $\Xi_{\xi} \in \mathfrak{X}(\cK)$
defined by $\Xi_{\xi}(k_{x}) = \frac{\partial}{\partial\varepsilon}\big|_{\varepsilon=0} k_{x}e^{-\varepsilon \xi(x)}$.
Using the equality $[\bv, \Xi_{\xi}] = \Xi_{D(\xi)}$, we write  
\begin{equation}
D (\xi) = L_{\bv}\xi\,.
\end{equation}
For $\fg = \Gamma_{c}(\fK)$, the Lie algebra $\fg \rtimes_{D}\R$ then has bracket
\begin{equation}\label{eq:smurfenliedje}
[\xi \oplus t, \xi' \oplus t'] = \Big( [\xi,\xi'] + (tL_{\bv}\xi' - t' L_{\bv}\xi)\Big) \oplus 0\,.
\end{equation}

\section{Reduction to simple Lie algebras}
\label{vanseminaarsimpel}

In this note, we will focus attention on the class of 
gauge algebras with a semisimple structure group, not only because they are more accessible, 
but also because they are relevant in applications. We now show that every gauge 
algebra with a \emph{semisimple} structure group can be considered as
a gauge algebra of a bundle with a \emph{simple} structure group 
which need not be the same for all fibers. 
Accordingly, the base manifold $M$ is replaced by a not necessarily connected finite 
cover.

\subsection{From semisimple to simple Lie algebras}
Let $\fK \rightarrow M$ be a smooth locally trivial bundle of Lie algebras with
semisimple fibres. 
We construct a finite cover $\widehat{M} \rightarrow M$ 
and a locally trivial bundle of Lie algebras
$\widehat{\fK} \rightarrow \widehat{M}$ with simple fibres 
such that $\Gamma(\fK) \simeq \Gamma(\widehat{\fK})$ and 
$\Gamma_{c}(\fK) \simeq \Gamma_{c}(\widehat{\fK})$.

Because one can go back and forth between principal fibre bundles and bundles of Lie 
algebras, this shows that every gauge algebra for a principal fibre bundle with 
semisimple structure group is isomorphic to one with a simple structure group.
Indeed, every principal fibre bundle $P\rightarrow M$ with semisimple structure group $K_{i}$
over the connected component $M_i$ of $M$ 
gives rise to the bundle $\Ad(P) \rightarrow M$ 
of Lie algebras. 
Conversely, every Lie algebra bundle $\fK \rightarrow M$ 
with semisimple structure algebra $\fk_{i}$ over the $M_i$ 
gives rise to a principal fibre 
bundle $P_{\fK} \rightarrow M$ with semisimple structure group $\mathrm{Aut}(\fk_{i})$ 
over $M_i$ defined, for $x \in M_i$,  by 
$P_{\fK,x} := \mathrm{Iso}(\fk_{i},\fK_{x})$ for $x \in M_i$. 

\begin{Theorem}\label{reductienaarsimpel} 
{\rm(Reduction from semisimple to simple structure algebras)}
If $\fK \rightarrow M$ is a smooth locally trivial bundle of Lie algebras
with semisimple fibres,
then there exists a
finite cover $\widehat{M} \rightarrow M$ and a smooth locally trivial bundle of Lie 
algebras $\widehat{\fK} \rightarrow \widehat{M}$ with simple fibres
such that there exist isomorphisms  
$\Gamma(\fK) \simeq \Gamma(\widehat{\fK})$ and 
$\Gamma_{c}(\fK) \simeq \Gamma_{c}(\widehat{\fK})$
of locally convex Lie algebras.
\end{Theorem}

The finite cover $\widehat{M} \rightarrow M$ is not necessarily connected, and 
the isomorphism classes of the fibres of $\widehat{\fK} \rightarrow \widehat{M}$ are not 
necessarily the same over different connected components of $\widehat{M}$.

\begin{proof}  
For a finite dimensional semisimple Lie algebra $\fk$, we write $\mathrm{Spec}(\fk)$ for the finite 
set  of maximal ideals of $\fk$, equipped with the discrete topology. 
We now define the set
\[
\widehat{M} := \bigcup_{x\in M} \mathrm{Spec}(\fK_{x})
\] 
with the natural projection $\pr_{\hat M} \colon \widehat{M} \rightarrow M$.
Local trivialisations $\fK|_{U} \simeq U \times \fk$ of $\fK$ 
over open connected subsets $U\subseteq M$
induce compatible bijections between
$\pr_{\hat M}^{-1}(U)$ and the smooth manifold $U \times \mathrm{Spec}(\fk)$.
This provides $\widehat{M}$ with a manifold structure for which 
$\pr_{\hat M} \colon \widehat{M} \rightarrow M$ is a finite 
covering.\footnote{%
Note that non-isomorphic maximal ideals of $\fK_{x}$
are always in different connected components of $\widehat{M}$, whereas 
isomorphic maximal ideals may or may not be in the same connected component,
depending on the bundle structure.}  
We define 
\[
\widehat{\fK} := \bigcup_{I_x\in \widehat{M}} \fK_{x}/I_{x}
\]
with the natural projection $\pi \colon \widehat{\fK} \rightarrow \widehat{M}$.
Local trivialisations $\fK|_{U} \simeq U \times \fk$ of $\fK$ 
yield bijections between $\widehat{\fK}|_{U}$
and the disjoint union 
\[\bigsqcup_{I\in \mathrm{Spec}(\fk)} U_{I} \times (\fk/I)\,,\]
where $U_{I} \simeq U$ is the connected component of $\pr_{\hat M}^{-1}(U)$ corresponding
to the maximal ideal $I\subseteq \fk$ in the particular trivialisation.
Since different trivialisations differ by Lie algebra automorphisms of the fibres,
which permute the ideals in $U_{I}$ and $\fk/I$ alike,
the projection $\pi \colon \widehat{\fK} \rightarrow \widehat{M}$ becomes a smooth locally trivial
bundle of Lie algebras over $\widehat{M}$.

The morphism 
$\Phi \colon \Gamma(\fK) \rightarrow \Gamma(\widehat{\fK})$ of Fr\'echet Lie algebras
defined by
\[
\Phi(\sigma)(I_{x}) := \sigma(x) + I_x \]
is an isomorphism; because the fibres are semisimple, the injection 
$\fK_{x}/I_{x} \hookrightarrow \fK_{x}$
allows one to construct the inverse
\[
\Phi^{-1}(\tau)(x) = \sum_{I_x \in \mathrm{Spec}(\fK_{x})}  \tau(I_{x})\,.
\]
Since the projection $\pr_{\hat M} \colon \widehat{M} \rightarrow M$ is a finite cover, this
induces an isomorphism
$\Phi \colon \Gamma_{c}(\fK) \rightarrow \Gamma_{c}(\widehat{\fK})$
of LF-Lie algebras.
\end{proof}


Clearly, a smooth 1-parameter family of automorphisms of $\fK \rightarrow M$ acts naturally 
on the maximal ideals, so we obtain a smooth action on $\widehat{M} \rightarrow M$
and on $\widehat{\fK} \rightarrow \widehat{M}$.
The action on $\widehat{M}$ is locally free or periodic if and only if the action on $M$ is,
and then the period on $\widehat{M}$ is a multiple of the period on $M$.

\begin{Example}If $\fk$ is a simple Lie algebra, then $\widehat{M} = M$.
\end{Example}
\begin{Example}
If $P = M \times K$ is trivial, then $\hat M = M \times \Spec(\fk)$ and all connected 
components of $\hat M$ are diffeomorphic to~$M$.
\end{Example}

\begin{Example}
If $\fk$ is a semisimple Lie algebra with $r$ simple ideals
that are mutually non-isomorphic, then
$\widehat{M} = \bigsqcup_{i=1}^{r} M$ is a disjoint union of copies of~$M$. 
\end{Example}

\begin{Example}(Frame bundles of 4-manifolds)\label{Ex:MhatOrientable} 
Let $M$ be a 4-dimensional Riemannian manifold. 
Let $P := {\rm OF}(M)$ be the principal $\OO(4,\R)$-bundle of orthogonal frames.
Then $\fk = \mathfrak{so}(4,\R)$ is isomorphic to 
$\mathfrak{su}_L(2,\C) \oplus \mathfrak{su}_R(2,\C)$.
The group $\pi_0(K)$ is of order 2, the non-trivial element acting by 
conjugation with $T = \mathrm{diag}(-1,1,1,1)$.
Since this permutes the two simple ideals, 
the manifold $\widehat{M}$ is the orientable double cover of $M$. This is 
the disjoint union $\widehat{M} = M_{L} \sqcup M_{R}$ of two copies of $M$ 
if $M$ is orientable, and a connected twofold cover $\widehat{M} \rightarrow M$ if it is not. 
%
\end{Example}

\subsection{Compact and noncompact simple Lie algebras}

A semisimple Lie algebra $\fk$ is called $\emph{compact}$ if its Killing form 
is negative definite. 
Every semisimple Lie algebra $\fk$ is a direct sum $\fk = \fk_{\mathrm{cpt}} \oplus \fk_{\mathrm{nc}}$,
where $\fk_{\mathrm{cpt}}$ is the direct sum of all compact ideals of $\fk$
(or, alternatively, its maximal compact quotient), and $\fk_{\mathrm{nc}}$
is the direct sum of the noncompact ideals.
Since the decomposition $\fk = \fk_{\mathrm{cpt}} \oplus \fk_{\mathrm{nc}}$ is invariant 
under $\mathrm{Aut}(\fk)$, every Lie algebra bundle bundle $\fK \rightarrow M$ 
can be written as a 
direct sum 
\begin{equation}\label{eq:dirsumcptnc}
\fK = \fK_{\mathrm{cpt}} \oplus \fK_{\mathrm{nc}}
\end{equation}
of Lie algebra bundles over $M$.
Correspondingly, we have the decomposition 
\begin{equation}
\widehat{M} = \widehat{M}_{\mathrm{cpt}} \sqcup \widehat{M}_{\mathrm{nc}}
\end{equation}
of $\widehat{M}$ into disjoint submanifolds, 
$\widehat{M}_{\mathrm{cpt}}$ and $\widehat{M}_{\mathrm{nc}}$,
containing
the maximal ideals $I_{x} \subset \fK_{x}$
with $\fK_{x}/I_{x}$ compact and noncompact, respectively.
Writing $\widehat{\fK}_{\mathrm{cpt}}$ 
for the restriction of $\widehat{\fK}$ to $\widehat{M}_{\mathrm{cpt}}$
and $\widehat{\fK}_{\mathrm{nc}}$ 
for its restriction to $\widehat{M}_{\mathrm{nc}}$, 
we find Lie algebra bundles $\widehat{\fK}_{\mathrm{cpt}} \rightarrow \widehat{M}_{\mathrm{cpt}}$ and 
$\widehat{\fK}_{\mathrm{nc}} \rightarrow \widehat{M}_{\mathrm{nc}}$ with compact 
and noncompact simple fibres respectively, and Fr\'echet Lie algebra isomorphisms 
\begin{equation}
\Gamma(\fK_{\mathrm{cpt}}) \simeq \Gamma(\widehat{\fK}_{\mathrm{cpt}}) \quad\text{and}\quad
\Gamma(\fK_{\mathrm{nc}}) \simeq \Gamma(\widehat{\fK}_{\mathrm{nc}})\,.
\end{equation}

\section{Universal invariant symmetric bilinear forms}
\label{sec:3.3} 

In Section~\ref{GySsCoc}, we will undertake a detailed analysis
of the 2-cocycles of $\fg \rtimes_{D} \R$
for compactly supported gauge algebras $\fg := \Gamma_{c}(\fK)$
with semisimple structure group $K$.
In order to describe the relevant 2-cocycles, we 
need to introduce universal invariant symmetric 
bilinear forms on the Lie algebra $\fk$ of the structure group.
In the case that $\fk$ is a compact simple Lie algebra, 
this is simply the Killing form.

\subsection{Universal invariant symmetric bilinear forms}\label{invbil}

Let $\fk$ be a finite dimensional Lie algebra.
Then its automorphism group $\Aut(\fk)$ is a closed subgroup 
of $\mathrm{GL}(\fg)$, hence a Lie group with Lie algebra $\der(\fk)$.
Since $\der(\fk)$ acts trivially on the quotient
\[V(\fk) := S^2(\fk)/{(\der(\fk) \cdot S^2(\fk))}\,\]
of the
twofold symmetric tensor power $S^2(\fk)$, 
the 
the $\Aut(\fk)$-representation on $V(\fk)$ factors through $\pi_0(\Aut(\fk))$.
The \emph{universal $\der(\fk)$-invariant symmetric bilinear form} is defined by
\[ \kappa \colon \fk \times \fk \rightarrow V(\fk), \quad 
\kappa(x,y) := [x\otimes_{s}y] = \frac{1}{2}[x \otimes y + y \otimes x].\]
We associate to
$\lambda \in V(\fk)^*$
the $\R$-valued, $\der(\fk)$-invariant, symmetric, bilinear form 
$\kappa_{\lambda} := \lambda \circ \kappa$.
This correspondence  
is a bijection between $V(\fk)^*$ and the space of $\der(\fk)$-invariant symmetric bilinear forms 
on $\fk$. 

For the convenience of the reader, we now list some properties of 
$V(\fk)$ for (semi)simple Lie algebras $\fk$, in which case $\der(\fk) = \fk$. 
These results will be used in the rest of the paper.
We refer 
to \cite[App.~B]{NW09} for proofs and a more detailed exposition.

For a simple real Lie algebra $\fk$,
we have $V(\fk) \simeq \K$, with $\K = \C$ if $\fk$ 
admits a complex structure, and $\K = \R$ if it does not, i.e., if $\fk$ is {\it absolutely simple}. 
The universal invariant symmetric bilinear form can be identified with 
the Killing form of the real Lie algebra $\fk$ if $\K = \R$ 
and the Killing form of the underlying complex Lie algebra if $\K = \C$.
In particular, in the important special case that $\fk$ is a compact simple Lie algebra,
the universal invariant bilinear form $\kappa \colon \fk \times \fk \rightarrow V(\fk)$
is 
simply the negative definite Killing form $\kappa \colon \fk \times \fk \rightarrow \R$, 
$\kappa(x,y) = {\rm tr}(\ad x \ad y)$.

For a semisimple real Lie algebra $\fk = \bigoplus_{i=1}^{r}\fk_{i}^{m_i}$, 
where the simple ideals $\fk_i$ are mutually non-isomorphic,  
one finds $V(\fk) \simeq \bigoplus_{i=1}^{r}V(\fk_i)^{m_i}$ with $V(\fk_i)$ 
isomorphic to $\R$ or $\C$. 
The 
action of $\pi_{0}(\Aut(\fk))$ on $V(\fk)$ leaves invariant the subspaces $V(\fk_i)^{m_i}$
coming from the isotypical ideals $\fk_i^{m_i}$. 
If $V(\fk_i)\simeq \R$, then
the action of $\pi_0(\Aut(\fk))$ on $V(\fk_i)^{m_i}$ factors through the homomorphism 
$\pi_{0}(\Aut(\fk)) \rightarrow S_{m_i}$
that maps $\alpha \in \Aut(\fk)$ to the permutation it induces on the set of ideals isomorphic 
to~$\fk_i$.
If $V(\fk_i)\simeq \C$, then
the action on $\C^{m_i}$ factors through a homomorphism 
$\pi_{0}(\Aut(\fk)) \rightarrow {(\Z/2\Z)^{m_i}\rtimes S_{m_i}}$, 
where the symmetric group $S_{m_i}$ acts by permuting components 
and $(\Z/2\Z)^{m_i}$ acts by complex conjugation in the components.

\subsection{The flat bundle $\bV = V(\fK)$} \label{flatbdl}

If $\fK \rightarrow M$ is a bundle of Lie algebras, we denote by $\bV \rightarrow M$
the vector bundle with fibres $\bV_{x} = V(\fK_{x})$. 
It carries a canonical flat 
connection $\dd$, defined by 
$\dd\kappa(\xi,\eta) := \kappa(\nabla\xi,\eta) + \kappa(\xi,\nabla\eta)$ for $\xi, \eta \in \Gamma(\fK)$,
where $\nabla$ is a \emph{Lie connection} on $\fK$, meaning that 
$\nabla [\xi,\eta] = [\nabla \xi , \eta] + [\xi, \nabla \eta]$ for all $\xi,\eta \in \Gamma(\fK)$.
As any two Lie connections differ by a $\der(\fK)$-valued 1-form, 
this definition is independent of the choice of $\nabla$ (cf.\ \cite{JW13}).


If $\fK$ has semisimple typical fibre $\fk$, then the 
isotypical ideals $\fk_{i}^{m_i}$ in the decomposition $\fk = \bigoplus_{i=1}^{r}\fk_{i}^{m_i}$
are $\Aut(\fk)$-invariant, so that we obtain a direct sum decomposition
\[\bV = \bigoplus_{i=1}^{r} \bV_i\]
of flat bundles.

If the ideal $\fk_i$ is absolutely simple, which is always the case 
if $\fk$ is a compact Lie algebra, then the structure group of $\bV_i$
reduces to $S_{m_i}$.
In particular, if $\fk$ is compact simple, then 
$\bV$ is simply the trivial line bundle $M \times \R \rightarrow M$.

If the ideal $\fk_i$ possesses a complex structure, then the 
structure group of $\bV_i$ reduces to 
$(\Z/2\Z)^{m_i}\rtimes S_{m_i}$.
In particular, for $\fk$ complex simple, the bundle $\bV \rightarrow M$ is the 
vector bundle with fibre $\C$, and $\alpha \in \Aut(\fk)$ flips 
the complex structure on $\C$ if and only if it flips the complex structure on $\fk$.
If $\fK = \Ad(P)$ for a principal fibre bundle $P \rightarrow M$ with complex simple structure group $K$, 
then $\bV$ is the trivial bundle $M \times \C \rightarrow M$.


\section{Central extensions of gauge algebras} 
\label{GySsCoc}

Let $\fg$ be 
the compactly supported gauge algebra $\Gamma_{c}(\fK)$
for a Lie algebra bundle $\fK \rightarrow M$ with semisimple fibres.
In this section, we will classify all possible central extensions of 
$\fg \rtimes_{D} \R$, in other words, we will
calculate the continuous second Lie algebra
cohomology $H^2(\fg \rtimes_{D} \R,\R)$ with trivial coefficients.
 In \cite{JN16} we will examine which of these cocycles comes 
 from a positive energy representation.

\begin{Remark}\label{Rk:perfectcocycles} 
For a cocycle $\omega$ on $\g \rtimes_D \R$, the relation 
\begin{equation}
  \label{eq:cocd}
\omega(D,[\xi,\eta]) = \omega(D\xi, \eta) + \omega(\xi,D\eta) 
\end{equation}
shows that $i_D \omega$ measures the non-invariance of 
the restriction of $\omega$ to $\g \times \g$ under the derivation~ $D$. 
It also shows that, if the Lie algebra $\g$ is perfect, then the linear functional 
$i_D\omega \: \g \to \R$ is completely determined by~\eqref{eq:cocd}. 
\end{Remark}

\subsection{Definition of the 2-cocycles}\label{subsec:2cocyc}
We define 2-cocycles $\omega_{\lambda,\!\nabla}$ on $\fg \rtimes_{D} \R$ such that their classes span the cohomology group $H^2(\fg \rtimes_{D} \R,\R)$.
They depend on a \emph{$\bV$-valued 
$1$-current} $\lambda \in \Omega^1_{c}(M,\bV)'$, and on a 
\emph{Lie connection} $\nabla$ on $\fK$.
%
%
Recall from Section \ref{sec:3.3}  
that $\kappa \colon \fk \times \fk \rightarrow V(\fk)$ is the universal
invariant bilinear form of $\fk$, and $\bV\rightarrow M$ is the flat bundle
with fibres $\bV_{x} = V(\fK_{x})$.
In the important special case that $\fk$ is compact simple, $V(\fk) = \R$, 
$\kappa$ is the Killing form, and $\bV$ is the trivial real line bundle.

A $1$-current $\lambda \in \Omega^1_{c}(M,\bV)'$ is said to be 
\begin{itemize}
\item[\rm(L1)]  {\it closed} if $\lambda(\dd C^\infty_c(M,\bV)) = 0$, and 
\item[\rm(L2)]  {\it $\pi_*\bv$-invariant} if 
$\lambda(L_{\pi_*\bv}\Omega^1_{c}(M,\bV)) = \{0\}$. 
\end{itemize}
Given a closed $\pi_*\bv$-invariant current $\lambda \in \Omega^1_{c}(M,\bV)'$, 
we define the
2-cocycle $\omega_{\lambda,\!\nabla}$ on $\g \rtimes_{D}\R$ 
by skew-symmetry and the equations
\begin{eqnarray}
\omega_{\lambda,\!\nabla}(\xi,\eta) &=& \lambda(\kappa(\xi,\nabla\eta)),\label{cdef1}\\
\omega_{\lambda,\!\nabla}(D,\xi) &=& \lambda(\kappa(L_{\bv}\nabla,\xi))\,, \label{cdef2}
\end{eqnarray}
where we write $\xi$ for $(\xi,0) \in \fg \rtimes_{D} \R$ and $D$ 
for $(0,1) \in \fg \rtimes_{D} \R$ as in \eqref{eq:d-elt}. 
We define the $\der(\fK)$-valued 1-form 
$L_{\bv}\nabla \in \Omega^1(M,\der(\fK))$ by 
\begin{equation} \label{eq:defcurv}
(L_{\bv}\nabla)_{w}(\xi) = L_{\bv} (\nabla\xi)_{w} - \nabla_{w}L_{\bv}\xi = 
L_{\bv}(\nabla_{w}\xi) - \nabla_{w}L_{\bv}\xi - \nabla_{[\pi_*\bv,w]} \xi
\end{equation}
for all $w \in \mathfrak{X}(M)$, $\xi \in \Gamma(\fK)$. Since the fibres of $\fK \rightarrow M$ are semisimple,
all derivations are inner, so we can identify $L_{\bv}\nabla$
with an element of $\Omega^1(M,\fK)$.
Using the formul\ae{} 
\begin{eqnarray}
\dd\kappa(\xi,\eta) &=& \kappa(\nabla \xi, \eta) + \kappa(\xi, \nabla \eta),\label{fijneformule1} \\
L_{\pi_*\bv}\kappa(\xi,\eta) &=& 
\kappa(L_\bv\xi,\eta) + \kappa(\xi, L_\bv\eta),\label{fijneformule2}\\
{}L_\bv(\nabla \xi) - \nabla L_\bv \xi &=& [L_{\bv}\nabla,\xi],\label{fijneformule3}
\end{eqnarray}
it is not difficult to check that $\omega_{\lambda,\!\nabla}$ is a cocycle. 
Skew-symmetry follows from (\ref{fijneformule1}) and (L1). 
The vanishing of
$\delta\omega_{\lambda,\!\nabla}$ on $\fg$ follows from
(\ref{fijneformule1}), the derivation property of $\nabla$ and 
invariance of $\kappa$. Finally, $i_{D}\delta\omega_{\lambda,\!\nabla} = 0$
follows from skew-symmetry, (\ref{fijneformule3}),
(\ref{fijneformule2}), (L2) and the invariance of $\kappa$.

Note that the class $[\omega_{\lambda,\!\nabla}]$ in $H^2(\fg\rtimes_{D}\R,\R)$
depends only on $\lambda$, not on $\nabla$. Indeed, 
two connection $1$-forms $\nabla$ and $\nabla'$ differ by $A \in \Omega^1(M,\der(\fK))$.
Using $\der(\fK) \simeq \fK$, we find
\[ \omega_{\lambda,\!\nabla'} - \omega_{\lambda,\!\nabla} = \delta \chi_{A}
\quad \mbox{ with } \quad 
\chi_{A}(\xi  \oplus t) := \lambda(\kappa(A,\xi)).\]

\subsection{Classification of central extensions}

We now show that every continuous Lie algebra $2$-cocycle on 
$\fg \rtimes_{D} \R$ is cohomologous to one of the type 
$\omega_{\lambda,\!\nabla}$ as defined in \eqref{cdef1} and \eqref{cdef2}. 
The proof relies on a description of $H^2(\fg,\R)$ provided 
by the following theorem (\cite[Prop.~1.1]{JW13}). 

\begin{Theorem} {\rm(Central extensions of gauge algebras)} \label{Ijkcykel}
Let $\fg$ be the compactly supported gauge algebra 
$\fg = \Gamma_c(\fK)$ of a Lie algebra bundle $\fK \rightarrow M$ with semisimple fibres.
Then every continuous 2-cocycle 
is cohomologous to one of the form
\[ \psi_{\lambda,\!\nabla}(\xi,\eta) = \lambda(\kappa(\xi, \nabla\eta)),\]
where $\lambda \colon \Omega_{c}^{1}(M,\bV) \rightarrow \R$ is a continuous
linear functional that vanishes on $\dd\Omega^{0}_{c}(M,\bV)$,
and $\nabla$ is a Lie connection on $\fK$.
Two such cocycles $\psi_{\lambda,\!\nabla}$ and $\psi_{\lambda',\!\nabla'}$ are equivalent 
if and only if $\lambda = \lambda'$.
\end{Theorem}

Using this, we classify the continuous central extensions of $\fg \rtimes_{D} \R$.
\begin{Theorem} {\rm(Central extensions of extended gauge algebras)}\label{EqIjkcykel}
Let $\cK\rightarrow M$ be a bundle of Lie groups with semisimple fibres, equipped with a 
1-parameter group of automorphisms with generator $\bv \in \mathfrak{X}(\cK)$. Let
$\fg = \Gamma_c(\fK)$ be the compactly supported gauge algebra, and let $\fg \rtimes_{D}\R$
be the Lie algebra \eqref{eq:smurfenliedje}.
Then the map $\lambda \mapsto [\omega_{\lambda,\!\nabla}]$ induces an isomorphism
\[\Big(\Omega^1_{c}(M,\bV) / \big(\dd\Omega^0_{c}(M,\bV) + L_{\pi_*\bv}\Omega^1_{c}(M,\bV)\big)\Big)'
\stackrel{\sim}{\longrightarrow} H^2(\fg\rtimes_{D}\R,\R) \] 
between the space of closed $\pi_*\bv$-invariant $\bV$-valued currents and 
$H^2(\g \rtimes_D \R,\R)$. 
\end{Theorem}

\begin{proof} 
Let $\omega$ be a continuous $2$-cocycle on $\fg \rtimes_{D}\R$. If 
$i \colon \fg \hookrightarrow \fg \rtimes_{D} \R$ is the inclusion, 
then $i^*\omega$ is a 2-cocycle on $\fg$.
By Theorem \ref{Ijkcykel} there exists a Lie connection $\nabla$ 
and a continuous linear functional $\phi \in \g'$ such that 
\[ i^*\omega(\xi,\eta) = \lambda(\kappa(\xi,\nabla\eta)) + \phi([\xi,\eta]), 
\quad \mbox{  where } \quad \lambda \in \Omega_{c}^1(M,\bV)'.\] 
Using the cocycle property (cf.\ Rk.~\ref{Rk:perfectcocycles}), we find
\begin{eqnarray}\label{DinG}
\omega (D, [\xi,\eta]) = 
 i^*\omega (L_{\bv}\xi , \eta) +  
 i^*\omega (\xi , L_{\bv}\eta )\,
\end{eqnarray}
and hence, using (\ref{fijneformule2})
and (\ref{fijneformule3}), 
\begin{eqnarray*}
\omega (D, [\xi,\eta]) 
 &=&
 \lambda\big(\kappa(L_{\bv}\xi, \nabla \eta) + \kappa(\xi,\nabla L_{\bv}\eta)\big)
 + \phi(L_{\bv}[\xi,\eta])\\
 &=&
 \lambda(L_{\pi_*\bv} \kappa(\xi,\nabla\eta)) + \lambda(\kappa(L_{\bv}\nabla, [\xi,\eta])) + \phi(L_{\bv}[\xi,\eta])\,.
\end{eqnarray*}
In particular, $[\xi,\eta] = 0$ implies 
$\lambda(L_{\pi_*\bv} \kappa(\xi,\nabla\eta)) = 0$.

Now fix a trivialisation $\cK|_{U} \simeq U \times K$ over an open subset $U\subseteq M$.
It induces the corresponding trivialisation $\bV|_{U} \simeq U \times V(\fk)$ of flat bundles.
For $f,g \in C^{\infty}_{c}(U)$ and $X\in \fk$, we 
consider $\xi = fX$ and $\eta = gX$ as commuting elements of $\Gamma_{c}(\fK)$.
With the local connection $1$-form $A \in \Omega^1(U,\fk)$, we then have
\[\kappa(\xi,\nabla \eta) = \kappa(fX,\dd g \cdot X + g[A,X]) =  f\, \dd 
g \cdot \kappa(X,X).\]  
Since $[\xi,\eta] = 0$, we find 
$\lambda\big((L_{\pi_*\bv} \beta \kappa(X,X)\big) = 0$ for all 
$1$-forms $\beta = f\dd g$ with $f,g \in C^\infty_c(U)$. 
As this holds for all $X\in \fk$ and as 
$\kappa(\fk,\fk) = V(\fk)$, we find
$\lambda(L_{\pi_*\bv}\Omega_{c}^1(U,\bV)) = \{0\}$ by polarisation.
Since 
$\Omega^1_{c}(M,\bV) = \sum_{i\in I} \Omega^1_{c}(U_i,\bV)$
for any trivialising open 
cover $(U_i)_{i\in I}$ of $M$, 
we find $\lambda(L_{\pi_*\bv}\Omega^1_{c}(M,\bV)) = \{0\}$.

Having established that $\lambda$ is $\pi_*\bv$-invariant, 
we may construct $\omega_{\lambda,\!\nabla}$ according to 
(\ref{cdef1}) and (\ref{cdef2}). 
It then follows from the above that the difference 
\[ \Delta \omega := \omega - \omega_{\lambda,\alpha} + \delta \phi^0,\] 
where $\phi^0$ is an extension of $\phi$ to $\g \rtimes_D \R$, 
satisfies $i^*\Delta\omega = 0$. 
Applying (\ref{DinG}) to $\Delta\omega$, we see that
$\Delta\omega(D,[\fg,\fg]) =0$ and hence that $\Delta\omega = 0$
because $\fg$ is perfect by \cite[Prop.~2.4]{JW13}.

This shows surjectivity of the map $\lambda \mapsto [\omega_{\lambda,\!\nabla}]$.
Injectivity follows because $\omega_{\lambda,\!\nabla} = \delta\chi$ implies
$\omega_{\lambda,\!\nabla}|_{\fg \times \fg} = \delta(\chi|_{\fg})$, hence
$\lambda = 0$ by Theorem \ref{Ijkcykel}.
\end{proof}

\begin{Remark}\label{Rk:curvatureisnice}
If the Lie connection $\nabla$ on $\fK$ can be chosen so as to make $\bv \in \mathfrak{X}(\cK)$ 
horizontal, $\nabla_{\pi_{*}\bv} \xi = L_{\bv}\xi$ for all $\xi \in \Gamma(\fK)$,
then equation \eqref{eq:defcurv} shows that
$L_{\bv}\nabla = i_{\pi_{*}\bv} R$, where $R$ is the curvature of $\nabla$.
%
For such connections, (\ref{cdef2}) is equivalent to 
\begin{equation}
  \label{eq:curv}
\omega_{\lambda,\!\nabla}(D, \xi ) = \lambda(\kappa(i_{\pi_*\bv}R,\xi)).
\end{equation}
\end{Remark}

%
%
%

%
%
%

\bibliographystyle{alpha}

\end{document}